\documentclass[12pt]{amsart}
\usepackage[top=30truemm,bottom=30truemm,left=30truemm,right=30truemm]{geometry}
\UseRawInputEncoding
\usepackage{mathrsfs}
\usepackage{bm}
\usepackage{amsfonts,amssymb}
\usepackage{dsfont}
\usepackage{extarrows}
\usepackage{amsmath}
\usepackage{enumerate}
\usepackage{amscd}
\usepackage[all,2cell]{xy}
\usepackage{hyperref}

\theoremstyle{plain}

\usepackage{tikz}
\usetikzlibrary{arrows.meta}
\theoremstyle{definition}

\newtheorem{thm}{Theorem}[section]

\newtheorem{prop}[thm]{Proposition}

\theoremstyle{definition}
\newtheorem{defn}{Definition}[section]

\theoremstyle{remark}
\newtheorem{rem}{\bf Remark}[section]

\newcommand{\be}{\begin{equation}}
\newcommand{\ee}{\end{equation}}
\newcommand{\bea}{\begin{eqnarray}}
\newcommand{\eea}{\end{eqnarray}}
\newcommand{\ben}{\begin{eqnarray*}}
	\newcommand{\een}{\end{eqnarray*}}
\newcommand{\bt}{\begin{split}}
	\newcommand{\et}{\end{split}}
\newcommand{\bet}{\begin{equation}}
\newcommand{\mc}{\mathbb{C}}
\newcommand{\mr}{\mathbb{R}}
\newcommand{\ra}{\rightarrow}
\newcommand{\beq}{\begin{equation*}}
\newcommand{\eeq}{\end{equation*}}
\newcommand{\bal}{\begin{aligned}}
\newcommand{\eal}{\end{aligned}}

\newcommand{\dbar}{\bar{\partial}}
\newcommand{\pa}{\partial}

\newcommand{\Ker}{{\textup{Ker}}}

\renewcommand{\leq}{\leqslant}%
\renewcommand{\geq}{\geqslant}%
\newcommand{\inner}[1]{\langle#1\rangle}
\newcommand{\iinner}[1]{\langle\langle#1\rangle\rangle}

\makeatletter
\g@addto@macro{\endabstract}{\@setabstract}
\newcommand{\authorfootnotes}{\renewcommand\thefootnote{\@fnsymbol\c@footnote}}%
\makeatother

\newcommand{\C}{\mathbb C}

\newcommand{\dist}{\operatorname{dist}}

\begin{document}
\title[]
{Pseudoconvex Exhaustions of Hartogs Domains Associated with $L^2$-Optimal Pairs I}
\author{Xujun Zhang}
\address{School of Mathematics, Guangxi University, Nanning, Guangxi, 530004, P.R.China}
\address{Guangxi Center for Mathematical Research, Guangxi University, Nanning, Guangxi, 530004, P. R. China}
\email{xujunzhang@amss.ac.cn, xujunzhang@gxu.edu.cn}
\date{}

\maketitle

\begin{abstract}
By constructing a smooth pseudoconvex exhaustion of the associated
Hartogs domain, we give a new proof of the Deng–Ning–Wang's characterization of
plurisubharmonic functions without the smoothness assumption. 
\end{abstract}

\section{Introduction}

A classical result of Grauert relates the negativity of a holomorphic line bundle to the existence of strongly pseudoconvex neighborhoods of its zero section \cite{Grauert1962}. By duality, this also provides a geometric interpretation of curvature positivity in terms of the geometry of neighborhoods of the zero section of the dual bundle.
Locally, the curvature positivity of a Hermitian line bundle is characterized by the plurisubharmonicity of its metric weight.
This naturally leads to studying positivity through pseudoconvex domains in the total space.

In \cite{DNW21,DWZZ,DNWZ22}, 
Deng, Ning, Wang, Zhang, and Zhou developed a converse $L^2$ theory by establishing analytic characterizations of plurisubharmonicity and Nakano positivity in terms of various $L^2$-conditions for the $\bar{\partial}$-operator.
In particular, they proved that a smooth Hermitian metric is Nakano semipositive if and only if it satisfies the ``\textit{optimal \(L^2\)-estimate condition}''. This characterization also leads to a notion of Nakano positivity for singular Hermitian metrics (\cite{In20}), and gives an affirmative answer to a question of Lempert (\cite{LYZ21,Liu-Xiao-Yang-Zhou}).

Let's recall the following definition.

\begin{defn}\label{def:L^2 peroperty}
Let $D$ be a domain in $\mc^n$ and
$\psi$  a real-valued continuous function defined on $ D$.
We say that the pair $(D,\psi)$ is \emph{$L^2$-optimal}
if for any smooth strictly plurisubharmonic function $\phi$ on $\mc^n$,
the equation $\bar\partial u=f$ can be solved on $D$ for any $\dbar$-closed $(0,1)$-form
$f \in L^2_{(0,1)}(D,\phi+\psi)$
with the estimate
$$
\int_D|u|^2 e^{-\phi-\psi} d\lambda
\leq
\int_D \langle B_\phi^{-1} f, f\rangle e^{-\phi-\psi} d\lambda,
$$
where
$B_\phi f = \sum^{n}_{i,j=1}  \frac{\partial^2\phi}{\partial z_i \partial\bar z_ j} f_j d\bar z_i$ and $d\lambda$ is the Lebesgue measure in $\mc^n$.
In particular,
we say that $D$ is \emph{$L^2$-optimal} if  the pair $(D,0)$ is \emph{$L^2$-optimal}.
\end{defn}

For a given pair \((D,\psi)\), we regard \(\psi\) as the weight of a Hermitian metric on the trivial line bundle over \(D\) and consider the associated Hartogs domain
\[
D_\psi
=
\left\{
(z,w)\in D\times\mathbb C:
|w|<e^{-\psi(z)}
\right\}.
\]
Geometrically, \(D_\psi\) is a fiberwise disc neighborhood of the zero section of the trivial line bundle
$D\times\mathbb C\longrightarrow D$
with fiber radius determined by \(\psi\). In this way, the plurisubharmonicity of \(\psi\) is reflected in the pseudoconvexity of the associated Hartogs domain.

This suggests the following geometric approach. Rather than deriving the plurisubharmonicity of \(\psi\) directly from the \(L^2\)-optimality condition, we show that \(L^2\)-optimality implies the pseudoconvexity of the associated Hartogs domain \(D_\psi\). The plurisubharmonicity of \(\psi\), and hence the corresponding curvature semipositivity, then follows from this pseudoconvexity.
The first main result of the present note is the following.

\begin{thm}\label{thm:main-1}
Let $D\subset\mathbb C^n$ be a bounded domain satisfying
$\mathring{\overline D}=D$, and let
$\psi\in C^0(D,\mathbb R)$.
If $(D,\psi)$ is $L^2$-optimal, then the associated Hartogs domain
\[
D_\psi
=
\{(z,w)\in D\times\mathbb C:|w|<e^{-\psi(z)}\}
\]
is pseudoconvex. In particular, $\psi$ is plurisubharmonic.
\end{thm}

Based on the preservation of \(L^2\)-optimality under restriction to complex submanifolds (see Proposition~3.3 in \cite{Liu-Xiao-Yang-Zhou}), the key step in the proof of Theorem~\ref{thm:main-1} is to show that the regularized function $\psi_\varepsilon
=
\textup{log} (e^\psi * \rho_\varepsilon)$
is plurisubharmoic on $D_\epsilon \subset D$.
Therefore, we can construct a pseudoconvex exhaustion of \(D_\psi\). 
The Behnke--Stein increasing-union theorem then proves the pseudoconvexity of \(D_\psi\).

\begin{rem}
This approach gives a new proof of the characterization of plurisubharmonic functions due to Deng--Ning--Wang under the sole assumption that \(\psi\) is continuous.
It therefore gives an affirmative answer to the following question posed by Deng, Ning, and Wang in \cite{DNW21}: ``\(\cdots\) it is natural to ask whether the regularity condition on \(\psi\) in Theorem~1.2 can be weakened to being continuous?''
\end{rem}

\begin{rem}
In recent work in \cite{Liu-25-1,Liu-25-2}, Liu gave an alternative approach to the question of Deng--Ning--Wang by establishing an optimal \(L^2\) extension theorem.
\end{rem}

\begin{rem}
The Hartogs-domain construction also has an interpretation in terms of singular Hermitian metrics. Since Nakano positivity and Griffiths positivity are equivalent for line bundles, the pseudoconvex exhaustion structure of \(D_{\varphi}\) gives a partial affirmative answer to a question posed by Inayama--Matsumura in \cite{Inayama-Matsumura}: ``Does the converse implication of Theorem~1.2 hold?'' More precisely, they asked whether singular Nakano positivity in the sense of \(L^2\)-estimates implies singular Nakano positivity in the sense of approximation.
In the line bundle case, the geometry of \(D_{\psi}\) gives a precise characterization of the latter notion.
\end{rem}

We next turn to the corresponding characterization of continuous convex functions in the real setting. We first recall the optimal estimate condition for the $d$-operator.
\begin{defn}\label{def:optimal $L^2$ estimate properties for function}
Let $\phi$ be a locally integrable real valued function on a domain $D\subset\mr^n$.
We say that $\phi$ satisfies \emph{the optimal $d-L^2$ estimate condition}, if for any $d$-closed smooth $1$-form $f=\sum^{n}_{i=1} f_{i}dx_i$ with compact support on $D$,
and any smooth strictly convex function $\psi$ on $D$, the equation $du=f$ can be solved on $D$ with the estimate
 $$
 \int_{D} |u|^2 e^{-\phi-\psi} d\lambda  \leq \int_{D}\sum^{n}_{i,j=1}\psi^{ij}f_{i}f_{j}e^{-\phi-\psi}   d\lambda,
 $$
 where $(\psi^{ij})_{n\times n}$ stands for the inverse of the matrix $(\psi_{ij})_{n\times n}$, with $\psi_{ij}=\frac{\pa^2 \psi}{\pa x_{i} \pa x_{j}}$.
\end{defn}

Motivated by the work of Deng-Ning-Wang and Deng-Ning-Wang-Zhou in \cite{DNW21} and \cite{DWZZ}
Deng-Zhang prove that any $C^2$-smooth function satisfying the optimal $d$-$L^2$ estimate condition must be convex (\cite{Deng-Zhang}). 
As an application of Theorem \ref{thm:main-1}, 
we extend this result to the continuous function, 
showing that the optimal $d$-$L^2$ estimate condition still implies convexity without the $C^2$-regularity assumption.

\begin{thm}\label{thm:convex}
Let $D$ be a bounded convex domain in $\mr^n$, and let $\phi:D\to\mathbb R$ be continuous. If $\phi$ satisfies the optimal $d$-$L^2$ estimate condition on $D$, then $\phi$ is convex.
\end{thm}

The proof is parallel to Theorem \ref{thm:main-1}, 
so we only state the main idea. For a standard
mollifier $\rho_{\epsilon}$, consider the log-exp regularization $\textup{log}(e^{\phi} * \rho_{\epsilon})$. 
Repeating the localization and $L^2$-duality argument in the real setting, 
the optimal $d-L^2$ estimate implies
that $p_{\epsilon}$ is smooth and convex on $D_{\epsilon}$. 
We omit the details of the proof.

\begin{rem}
The approach developed in the proof of Theorem~\ref{thm:main-1} also suggests a natural extension to the positivity of Hermitian metrics on holomorphic vector bundles. 
More precisely, 
let \(E\) be a holomorphic vector bundle over a Stein manifold \(X\), equipped with a continuous Hermitian metric \(h\). If \((E,h)\) satisfies the optimal \(L^2\)-estimate condition introduced in \cite{DWZZ}, 
then the associated dual unit disk bundle
\[
\Omega_h
=
\left\{
(x,\xi)\in E^*:
\|\xi\|_{h^*}^2<1
\right\}
\]
admits a pseudoconvex exhaustion. Moreover, \((E,h)\) is Griffiths semipositive.

This give a new approach to Question 7.3. in \cite{In20}.
Since the present note focuses on the line bundle condition, 
we will give a detailed of the above resluts in the furthercoming work.
\end{rem}

The rest of this note is organized as follows. 
In \S~ \ref{sec:log-exp-regularization}, we establish the plurisubharmonicity of the log-exp regularizations by combining restriction to complex lines with a one-dimensional dual estimate. In \S~\ref{sec:4}, we use a decreasing sequence of smooth plurisubharmonic approximations to construct an increasing pseudoconvex exhaustion of the associated Hartogs domain and complete the proof of Theorem~\ref{thm:main-1}.

\textbf{Acknowledgement:}
The author is grateful to Professor Fusheng Deng, his Ph.D. advisor, for introducing him to the \(L^2\) converse theory and for his valuable suggestions on related topics. 
The authors is also grateful to Professor Zhuo Liu to draw author's attention to the Hartogs domain in the previous joint work in \cite{Liu-Zhang}.
The author was supported by the National Natural Science Foundation of China (No.12501102), Guangxi Natural Science Foundation (Young Talents Progrem 2026) and the Tianyuan Mathematical Center in Southwest China, Guangxi Base.

\section{log-exp regularization for $L^2$-optimal function}\label{sec:log-exp-regularization}

\subsection{Notation and weighted $L^2$-spaces}

We first fix the notation for weighted $L^2$-spaces and the $\bar\partial$-operator used throughout this note.

Let $D$ be a domain in $\mc^n$.
We denote by $\wedge^{0,q}(D)$ the space of  $(0,q)$-forms on $D$, for any $0\leq q\leq n$ ((0,0)-forms are just smooth functions), and
$\wedge^{0,q}_c(D)$ the elements in
$\wedge^{0,q}(D)$ with compact support.
Let $\varphi$ be a real-valued continuous function on $D$.
Given $\alpha=\sum_I\alpha_Id\bar z_I, \beta=\sum_I\beta_Id\bar z_I\in \wedge^{0,q}(D)$, we define the products of $\alpha$ and $\beta$ and the corresponding norm with respect to $\varphi$ as follows:
$$
\bal
\inner{\alpha, \beta}_{\varphi} &= \sum_I \alpha_I \cdot \overline\beta_I e^{-\varphi},
|\alpha|^2_{\varphi}=\inner{\alpha,\alpha}_{\varphi}. \\
\iinner{\alpha, \beta}_{\varphi} &=\int_D \inner{\alpha,\beta}e^{-\varphi} d\lambda,
\|\alpha \|^2_{\varphi}=\iinner{\alpha,\alpha}_{\varphi}.
\eal
$$
Let $L^2_{(0,q)} (D,\varphi)$ be the completion of $\wedge^{0,q}(D)$ with respect to the inner product $\iinner{\cdot,\cdot}_{\varphi}$,
then $L^2_{(0,q)}(D,\varphi)$ is a Hilbert space and
$\dbar:L^2_{(0,q)}(D,\varphi)\ra L^2_{(0,q+1)}(D,\varphi)$ is a  closed and densely defined operator.
Let $\dbar_\varphi^{*}$ be the Hilbert adjoint of $\dbar$.
For convenience,
we denote $L^2 (D, \varphi):=L^2_{(0,0)} (D, \varphi)$.

\subsection{Plurisubharmonic functions and log-exp regularization}

We begin by recalling the definitions of plurisubharmonicity and mollification.

\begin{defn}
	Let $\Omega$ be a domain in $\mc^n$
A function \(u : \Omega \longrightarrow [ - \infty , + \infty ) \) defined on an open subset \(\Omega \subset \mathbb { C } ^ { n }\) is said to be plurisubharmonic (psh for short) if

\begin{enumerate}[(i)]
\item
  u is upper semicontinuous ;
\item
  for every complex line \(L \subset \mathbb { C } ^ { n }\),
 \(u\big| _ { \Omega \cap L }\) is subharmonic on \(\Omega \cap L\) , that is, for all \(a \in \Omega\) and \(\xi \in \mathbb { C } ^ { n }\) with \(| \xi | < d ( a , \pa \Omega )\) , the function u satisfies the mean value inequality:
\end{enumerate}

\[
u (a) \leqslant \frac {1}{2 \pi} \int_ {0} ^ {2 \pi} u (a + e ^ {\mathrm{i} \theta} \xi) d \theta .
\]

The set of psh functions on \ensuremath{\Omega} is denoted by Psh(\ensuremath{\Omega}).
\end{defn}

\begin{defn}
Let $m\geq1$. A \emph{mollifier} is a nonnegative function
$\rho\in C_c^\infty(\mathbb R^m)$ satisfy the following:
\begin{enumerate}
  \item $\operatorname{supp}\rho\subset B(0,1)$
  \item $\displaystyle \int_{\mathbb R^m}\rho(x)dx=1$.
\end{enumerate}
For $\varepsilon>0$, 
set $\rho_\varepsilon(x)
=
\varepsilon^{-m}\rho\left(\frac{x}{\varepsilon}\right).$
The family $\{\rho_\varepsilon\}_{\varepsilon>0}$ is called the
associated family of mollifiers, or an approximate identity.

If $\Omega\subset\mathbb R^m$ is a domain and
$f\in L^1_{\mathrm{loc}}(\Omega)$, the $\varepsilon$-mollification of
$f$ is defined by
\[
f_\varepsilon(x)
:=
(f*\rho_\varepsilon)(x)
=
\int_{\mathbb R^m}f(x-y)\rho_\varepsilon(y)dy,
\]
whenever $B(x,\varepsilon)\subset\Omega$.
\end{defn}

Inspired by Theorem 5.6 in Chapter 1 of \cite{Demaily},
one can prove 
the following log-exponential regularization
result for plurisubharmonic functions.
Since the argument is standard, we omit the proof.

\begin{prop}
\label{prop:exp-regularization-org}
Let $\Omega$ be a domain in $\mc^n$,
and let
$u\in\operatorname{Psh}(\Omega)$ with $u\not\equiv-\infty$.
Let $\{\rho_\varepsilon\}_{\varepsilon>0}$ be a family of
standard nonnegative  mollifiers. 
Set
$\Omega_\varepsilon
:=\{z\in\Omega:\operatorname{dist}(z,\partial\Omega)
>\varepsilon\}$ and
$u_\varepsilon:=\log(e^u*\rho_\varepsilon)$
on $\Omega_\varepsilon$.
Then $u_\varepsilon$ is smooth and plurisubharmonic
on $\Omega_\varepsilon$. Moreover,
$u_\varepsilon$ converges pointwise to $u$ on $\Omega$
as $\varepsilon\to0$.
\end{prop}

\begin{rem}
Proposition~\ref{prop:exp-regularization-org} shows that
log-exponential regularization preserves plurisubharmonicity.
\end{rem}

\subsection{Restriction to complex lines and dual inequality}

We now turn to an $L^2$-optimal pair $(D,\psi)$.
First,
we recall the following restriction property, which is a special case of Proposition~3.3 in \cite{Liu-Xiao-Yang-Zhou}.

\begin{prop}\label{prop:line-localization}
Let $D$ be a bounded domain in $\mc^n$, 
$\psi$ be a real-valued continuous function defined on $D$.
Assume the pair $(D,\psi)$ is $L^2$-optimal,
then 
$(L,\psi|_L)$ is $L^2$-optimal for any complex disc $L$ contained in $D$.
\end{prop}

\begin{proof}
Any complex line disc \(L\subset D\) can be written as
\[
L
=
\left\{
z_0+\zeta v : |\zeta|<r
\right\},
\]
where \(z_0\in D\), \(r>0\), and \(v\in\mathbb C^n\) is a unit vector, i.e.,$|v|=1.$
After a unitary affine change of coordinates, we may assume
\[
z=(\zeta,z')\in \mathbb C\times \mathbb C^{n-1},
\qquad
L=\mathbb D_r\times\{0\},
\]
where $\mathbb D_r:=\{\zeta\in\mathbb C:|\zeta|<r\}.$
Set $\psi_L(\zeta)=\psi(\zeta,0).$

For any $(0,1)$-form $g(\zeta)d\bar\zeta
\in \wedge^{0,1}_{c}(L)$,
since $\dim_{\mc} L=1$, we have $g(\zeta)d\bar\zeta \in \Ker (\dbar)$.
Extend \(g\) by zero to a smooth compactly supported function on \(\mathbb C\). Then there exists \(a\in C^\infty(\mathbb C)\) satisfy  
$$
\bar\partial a=g(\zeta)d\bar\zeta.
$$
Since \(L\subset D\), we can choose $r<r_0<r_1<r_2,
0<\delta_1<\delta_2,$
such that $\overline{\mathbb D}_{r_2}
\times
\overline{B}_{\delta_2}^{n-1}
\subset D.$
Choose cut-off functions $\chi_1\in C_c^\infty(\mathbb D_{r_2}),
\chi_2\in C_c^\infty(B_{\delta_2}^{n-1}),$
satisfying
\[
\chi_1\equiv1
\quad\text{on }\mathbb D_{r_1},
\qquad
\chi_2\equiv1
\quad\text{on }B_{\delta_1}^{n-1},
\]
and set $\chi(\zeta,z')
:=
\chi_1(\zeta)\chi_2(z').$
Then the function $a$ can be viewed as a function
on $\mc^n$ that is independent of $z'$.
Set $G(\zeta,z')
=
\chi(\zeta,z')a(\zeta),
F=\bar\partial G.$
Then \(F\) is a smooth, compactly supported, \(\bar\partial\)-closed \((0,1)\)-form on \(D\)
and $F
=
g(\zeta)d\bar\zeta
\;
\text{on} \;
\mathbb D_r\times B_{\delta_1}^{n-1}.$

Choose a smooth, convex, nondecreasing function
$\tau:[0,\infty)\to[0,\infty)$ such that
$\tau(t)=0$ for $0\le t\le r_0^2$ and
$\tau(t)\ge 1$ for $t\ge r_1^2$.
Set $q(\zeta,z')
:=
|z'|^2+\tau(|\zeta|^2)$，
then \(q\) is
plurisubharmonic on $\mc^n$ and $q(\zeta,z')=|z'|^2$ for $|\zeta|<r.$ since $r<r_0$.

For any smooth strictly subharmonic function $\varphi$ defined on $\mc$,
set 
$\Phi_k(\zeta,z')=\varphi(\zeta)+kq(\zeta,z')$ for any $k \in \mathbb{N}^{*}$,
then $\Phi_k$ is a smooth strictly plurisubharmonic function on
$\mc^n$. 

Since $(D,\psi)$ is $L^2$-optimal, 
there exists
$U_k\in L^2(D,\Phi_k+\psi)$ such that 
$\dbar U_k=F$
and
\begin{equation}\label{eq:ambient-optimal}
\int_D
|U_k|^2e^{-\Phi_k-\psi}d\lambda
\leq
\int_D
\left\langle
B_{\Phi_k}^{-1}F,F
\right\rangle
e^{-\Phi_k-\psi}d\lambda.
\end{equation}
Since $F=\dbar G$, the distribution $U_k-G$ is holomorphic on $D$.
After changing $U_k$ on a $L^2$-negligible set if necessary,
we may assume that
$U_k$ is smooth.

Since $F=g(\zeta)d\bar\zeta$, $q=|z'|^2$
on $\operatorname{supp}g\times B_{\delta_1}^{n-1}$ 
and $\varphi$ is strictly subharmonic.
There exists a constant $C>0$,
independent of $k$, such that
\begin{equation}\label{eq:B-inverse-main}
\left\langle
B_{\Phi_k}^{-1}F,F
\right\rangle
=
\frac{|g(\zeta)|^2}
{\varphi_{\zeta\bar\zeta}(\zeta)} \leq C|F|^2 <+\infty.
\end{equation}

 Since \(F\) is smooth, compactly supported,
 $D$ is bounded and  
 the continuity of
$\psi$,
the integration tends to $0$
outside $\mathbb D_r\times B_{\delta_1}^{n-1}$,
thus there exists $c_0>0$ such that 
\begin{equation}\label{eq:Rk-main}
  \begin{aligned}
  \int_D
\left\langle
B_{\Phi_k}^{-1}F,F
\right\rangle
e^{-\Phi_k-\psi}d\lambda
&=
\int_{\mathbb D_r\times B_{\delta_1}^{n-1}}
\frac{|g(\zeta)|^2}
{\varphi_{\zeta\bar\zeta}(\zeta)}
e^{-\varphi(\zeta)-\psi(\zeta,z')-k|z'|^2}
d\lambda(\zeta)d\lambda(z')\\
&
+O(e^{-c_0k}).
  \end{aligned}
\end{equation}
Set 
\[
c_k:=
\int_{|z'|<s_k}
e^{-k|z'|^2}d\lambda(z'),
\]
where $s_k:=k^{-1/4}.$
Then $s_k\to0$ and $c_k\sim\left(\frac{\pi}{k}\right)^{n-1}.$
Then by the Fubini's theorem,
for sufficiently large $k$, $s_k<\delta_1$,
$|\xi|<s_k$,
we have 
$$
\begin{aligned}
\int_{\mathbb D_r\times B_{\delta_1}^{n-1}} &|U_{k}|^2
e^{-\varphi(\zeta)-\psi(\zeta,z')-k|z'|^2}
d\lambda(\zeta)d\lambda(z') \\
& =\int_{|\xi|<s_k}e^{-k|\xi|^2}d\lambda(\xi)
\int_{\mathbb D_r}
|U_k(\zeta,\xi)|^2
e^{-\varphi(\zeta)-\psi(\zeta,\xi)}
d\lambda(\zeta) \\
& \leq \int_D  \left\langle
B_{\Phi_k}^{-1}F,F
\right\rangle
e^{-\Phi_k-\psi}d\lambda.
\end{aligned}
$$
Thus we can find 
$\xi_k\in\mc^{n-1}$ with $|\xi_k|<s_k$
such that
\begin{equation}\label{eq:slice-estimate}
\int_{\mathbb D_r}
|U_k(\zeta,\xi_k)|^2
e^{-\varphi(\zeta)-\psi(\zeta,\xi_k)}
d\lambda(\zeta)
\leq \frac{1}{c_k}.\int_D  \left\langle
B_{\Phi_k}^{-1}F,F
\right\rangle
e^{-\Phi_k-\psi}d\lambda.
\end{equation}
Since $\psi$ is continuous on the compact set $\overline{\mathbb D}_r\times\overline{B}_{\delta_1}^{n-1},$
by \eqref{eq:Rk-main},
we have 
\begin{equation}\label{eq:Rk-limit}
\lim_{k\to\infty}\frac{1}{c_k}.\int_D  \left\langle
B_{\Phi_k}^{-1}F,F
\right\rangle
e^{-\Phi_k-\psi}d\lambda.
=
\int_{\mathbb D_r}
\frac{|g(\zeta)|^2}
{\varphi_{\zeta\bar\zeta}(\zeta)}
e^{-\varphi(\zeta)-\psi_L(\zeta)}
d\lambda(\zeta).
\end{equation}

Set $u_k(\zeta):=U_k(\zeta,\xi_k),
\zeta\in\mathbb D_r.$
Since $|\xi_k|<\delta_1$ for large $k$, we have $\dbar u_k=f
$ on $\mathbb D_r.$

Hence $u_k-a \in \Ker (\dbar)$
 on $\mathbb D_r$,
by \eqref{eq:slice-estimate} and \eqref{eq:Rk-limit}, 
the sequence $\{u_k\}$ is bounded
in $L^2(\mathbb D_r)$.
Thus there exists $h\in\mathcal O(\mathbb D_r)$ such that 
$u_k-a\longrightarrow h$
uniformly on compact subsets of $\mathbb D_r$.

Set $u=a+h$
we have $\dbar u=g(\zeta)d\bar\zeta$,
by Fatou's lemma,
\begin{align*}
\int_{\mathbb D_r}
|u|^2e^{-\varphi-\psi_L}d\lambda
&\leq
\liminf_{k\to\infty}
\int_{\mathbb D_r}
|u_k|^2
e^{-\varphi-\psi(\zeta,\xi_k)}
d\lambda\\
&\leq \int_{\mathbb D_r}
\frac{|g(\zeta)|^2}
{\varphi_{\zeta\bar\zeta}(\zeta)}
e^{-\varphi(\zeta)-\psi_L(\zeta)}
d\lambda(\zeta).
\end{align*}
This proves $L^2$ optimal properties for
$g\in \wedge^{0,1}_{c}(L)$.

Furthermore,
for any $f\in L^2_{(0,1)}(L,\varphi+\psi_L)$
be arbitrary, satisfying
\[
\int_L
\langle B_\varphi^{-1}f,f\rangle
e^{-\varphi-\psi_L}d\lambda
<+\infty.
\]
We can choose $
f_\nu\in \wedge^{0,1}(D)$
converging to $f$ such that 
there exist $u_\nu$ satisfying
$\dbar u_\nu=f_\nu$
and
\[
\int_L
|u_\nu|^2e^{-\varphi-\psi_L}d\lambda
\leq
\int_L
\langle B_\varphi^{-1}f_{\nu},f_\nu \rangle
e^{-\varphi-\psi_L}d\lambda
\]

Thus $\{u_\nu\}$ is bounded in
$L^2(L,\varphi+\psi_L)$. Then there exists $u$ such that 
$u_\nu\rightharpoonup u$ with $\dbar u=f$  
satisfying
\[
\int_L
|u|^2e^{-\varphi-\psi_L}d\lambda
\leq
\int_L
\langle B_\varphi^{-1}f,f\rangle
e^{-\varphi-\psi_L}d\lambda.
\]
Therefore $(L,\psi|_L)$ is $L^2$-optimal.
\end{proof}

Next, we get a one-dimensional dual inequality
involving the exponential of the $L^2$-optimal weight.

\begin{prop}\label{prop:one-dimensional-dual}
Let $D$ be a bounded domain in $\mc^n$, 
$\psi$ be a real-valued continuous function defined on $D$.
Assume that $(D,\psi)$ is $L^2$-optimal,
then for any smooth strictly
subharmonic function $\varphi$ on $\mc$, 
for any complex disc $L$ contained in $D$,
we have
\begin{equation}\label{eq:dual}
\int_L
(|\dbar \alpha |^2-\Delta \varphi \cdot |a|^2 )e^{\varphi+\psi_{L}}d\lambda.
\end{equation}
holds for any
$a\in C_c^\infty(L)$, 
where $\psi_L=\psi |_{L}$.

\end{prop}

\begin{proof}
Since $\dim_{\mc}L=1 $,
the $(0,1)$-form $f=
\Delta \varphi \cdot e^{\varphi+\psi_{L}}\cdot  ad\bar\zeta$
is $\dbar$-closed on $L$
For any smooth strictly
subharmonic function $\varphi$ on $\mc$
and any  $a\in C_c^\infty(L)$.
Therefore,
there exits $u \in L^2(L,\varphi+\psi _{L})$, satisfy $\dbar u=f$ with the estimate
\begin{equation}\label{eq:main-ineq-dim-1}
\begin{aligned}
\int_{L} |u|^2 e^{-\varphi-\psi_{L}} d\lambda 
&\leq \int_{L}\frac{|\Delta \varphi \cdot  e^{\varphi+\psi_{L}}ad\bar\zeta|^2}{\Delta \varphi}  e^{-\varphi-\psi_{L}}  d\lambda\\
& =\int_{L} \Delta \varphi \cdot |a|^2 e^{\varphi+\psi_{L}}  d\lambda <+\infty
\end{aligned}
\end{equation}
Since \(\bar\partial u=\Delta\varphi\cdot e^{\varphi+\psi_L}a\,d\bar\zeta\), 
identify $\bar\partial u$ with its coefficient
in the coordinate $\zeta$ and pairing the identity
$\bar\partial u=\Delta\varphi\,
e^{\varphi+\psi_L}a\,d\bar\zeta$
with $a$, 
we have
\[
\begin{aligned}
\int_{L} \Delta \varphi \cdot |a|^2 e^{\varphi+\psi_{L}}  d\lambda
&=
\int_L
\dbar u \cdot \overline a d\lambda
\\
&=
-\int_L
u \cdot \overline{\dbar \alpha } d\lambda.
\end{aligned}
\] 
where the last equality follows from integration by parts, since \(a\in C_c^\infty(L)\).

By the Cauchy--Schwarz inequality,
we have 
\[
\begin{aligned}
\left( \int_{L} \Delta \varphi \cdot |a|^2 e^{\varphi+\psi_{L}}  d\lambda\right)^2
&\leq
\left(
\int_{L}
|u|^2e^{-\varphi-\psi_{L}}d\lambda
\right)
\left(
\int_L
|\dbar \alpha|^2e^{\varphi+\psi_{L}}d\lambda
\right) \\
&\leq
\left( \int_{L} \Delta \varphi \cdot |a|^2 e^{\varphi+\psi_{L}}  d\lambda\right)
\left(
\int_L
|\dbar \alpha |^2e^{\varphi+\psi_{L}}d\lambda
\right),
\end{aligned}
\]
Then we have 
$$
\int_{L}
(|\dbar \alpha |^2-\Delta \varphi \cdot |a|^2 )e^{\varphi+\psi_{L}}d\lambda \geq 0,
$$
holds for any
$a\in C_c^\infty(L)$.

\end{proof}

\subsection{Log-Exponential Regularization}

We now apply Proposition~\ref{prop:one-dimensional-dual} to the log-exp regularization of the $L^2$ optimal weight.

\begin{prop}\label{prop:main-mollifier}
Let $D$ be a bounded domain in $\mathbb{C}^n$ 
and $\psi$ be a real-valued continuous function defined on $D$.
Assume that $(D,\psi)$ is $L^2$-optimal,
then for any stander mollifier $\rho_{\epsilon}$,
the function 
$$
\psi_\epsilon =\textup{log} (e^{\psi} * \rho_{\epsilon})
$$
converge to $\psi$
and plurisubharmonic on $D_\varepsilon$
where 
$D_\varepsilon
:=
\{z\in D:\dist(z,\C^n\setminus D)>\varepsilon\}.$
\end{prop}

\begin{proof}
For any complex disc $L$ contained in $D$,
Proposition \ref{prop:line-localization} shows that $(L,\psi\big|_{L})$ is $L^2$ optimal.
The stander regularization technique prove the convergence part.
The remaining is to show that $\rho_\epsilon$ is subharmonic on any complex disc $L$ contained in $D$.

Fix $z_0\in D_\epsilon$ and a unit vector $v\in\C^n$. Choose
$r>0$ sufficiently small such that 
\[
L
=
\{z_0+\zeta v:|\zeta|<r\}
\subset D_\varepsilon \subset D.
\]
Set $\mathbb D_r:=\{\zeta\in\C:|\zeta|<r\}$, 
we can choose a propert translation $L-\xi=\mathbb D_r \subset D_{\epsilon}$,
then by Proposition \ref{prop:one-dimensional-dual}, we have 
$$
\int_{\mathbb D_r}
(|\dbar \alpha |^2-\Delta \varphi \cdot |a|^2 )e^{\varphi+\psi_{L}(\zeta)}d\lambda \geq 0
$$
holds for any $\alpha \in C^{\infty}_{c}(\mathbb D_{r})$.
Multiplying  $\rho_\varepsilon(\xi)$ and integrating in $\xi$ shows that 
$$
\int_{\mc^n} \rho_\varepsilon(\xi) \int_{\mathbb D_{r}}
(|\dbar \alpha |^2-\Delta \varphi \cdot |a|^2 )e^{\varphi+\psi_{L}(\zeta)}d\lambda d\mu \geq 0,
$$
by the Fubini's theorem,
we have 
\begin{equation}\label{eq:333}
	\int_{\mathbb D_r}
\left(
|\dbar \alpha|^2
-
\Delta \varphi \cdot |a|^2
\right)
e^{\varphi(\zeta)}
 e^{\psi_{L} (\zeta)} *\rho_{\epsilon}
d\lambda
\geq0.
\end{equation}

Set 
$p(\zeta)
=
 \textup{log} (e^{\psi_{L} (\zeta)} *\rho_{\epsilon})$, 
$q=\varphi+p(\zeta)$, 
for any $\chi \in C_c^\infty(\mathbb D_r)$, 
set $a=e^{-\varphi-p(\zeta)}   \chi$,
then we have
\[
\dbar a
=
e^{-\varphi-p(\zeta)}(\dbar  \chi -\dbar q \cdot \chi).
\]
substituting this into \eqref{eq:333},
we have 
\begin{equation}\label{eq:before-weighted-identity}
	\int_{\mathbb D_r}
\left(
|\dbar  \chi -\dbar q \cdot \chi|^2
-
\Delta \varphi \cdot |\chi|^2
\right)
e^{-\varphi-p(\zeta)} 
d\lambda
\geq0.
\end{equation}
Since 
$$
\begin{aligned}
\overline{\dbar \chi-\dbar q \cdot \chi}e^{-q}
&=
\left(
\pa  \overline{\chi}
-\pa q \bar\chi
\right)e^{-q}  \\
&=
\partial
\left(\bar\chi e^{-q}\right).
\end{aligned}
$$
we have 
$$
\begin{aligned}
\int_{\mathbb D_r}
|\dbar  \chi -\dbar q \cdot \chi|^2 e^{-q } d\lambda
&=
\int_{\mathbb D_r}
(\dbar  \chi -\dbar q \cdot \chi)
\partial
\left(\bar\chi e^{-q}\right) d\lambda
\end{aligned}
$$
Since \(\chi\) has compact support in \(\mathbb D_r\), 
we have 

$$
\begin{aligned}
\int_{\mathbb D_r}
|\dbar  \chi &-\dbar q \cdot \chi|^2 e^{-q } d\lambda\\
&=
-
\int_{\mathbb D_r}
\pa (\dbar  \chi -\dbar q \cdot \chi)\bar\chi e^{-q}
 d\lambda\\
&=-\int_{\mathbb D_r} (\pa \dbar \chi -\pa \dbar \varphi-\pa \dbar p(\zeta) \cdot \chi -\dbar q \cdot \pa \chi) \bar\chi e^{-q} d\lambda \\
&=-\int_{\mathbb D_r} \Delta \chi \cdot  \bar\chi e^{-q} d\lambda+ \int_{\mathbb D_r} \Delta \varphi \cdot  \bar\chi e^{-q} d\lambda+\int_{\mathbb D_r} \Delta p(\zeta) \cdot  \bar\chi e^{-q} d\lambda \\
&\;\;\;\;+\int_{\mathbb D_r} \dbar q \cdot \pa \chi\cdot  \bar\chi e^{-q} d\lambda
\end{aligned} 
$$
integration by parts, we have 
$$
\begin{aligned}
-\int_{\mathbb D_r} \Delta \chi \cdot  \bar\chi e^{-q} d\lambda
&=
\int_{\mathbb D_r}
\dbar \chi \cdot 
\overline{\pa(\bar\chi e^{-q})}d\lambda \\
&=\int_{\mathbb D_r}
\dbar \chi \cdot \overline{(\pa \bar \chi -\pa q\bar \chi )} e^{-q}
d\lambda \\
&=\int_{\mathbb D_r}
\dbar \chi \cdot (\dbar   \chi -\dbar  q\pa  \chi ) e^{-q}
d\lambda
\end{aligned}
$$
we obtain
\begin{equation}\label{eq:levi-test}
\int_{\mathbb D_r}
\left(
|\dbar \chi|^2
+
\Delta p \cdot |\chi|^2
\right)
e^{-\varphi-p}d\lambda
\geq0.
\end{equation}
holds for any $\chi \in C_c^\infty(\mathbb D_r)$.
The remaining is to show that $p$ is subharmonic function.

We now argue by contradiction. Suppose that $p$ is not subharmonic on $\mathbb{D}_r$, 
then there is $z_0\in \mathbb{D}_r$, $r_1>0$, a constant $c>0$, and a contant $\xi\in
\mathbb{C}$ with $|\xi|=1$,  
such that
$$
\left(\frac{\partial^2 p}{\partial \zeta \partial\bar{\zeta}}(z)\right)\xi\bar{\xi}<-c
$$
holds for any $z\in B(z_0,r_1):=\{z\in \mc: |z-z_0|<r_1\}\subset \mathbb{D}_r$.
We may assume that $z_0=0$, and write $B(0,r_1)$ as $B_{r_1}$.

Choose $\kappa \in C^\infty_c(B_{r_1})$, satisfying $\kappa(z)=1$ for $z\in B_{r_1/2}$.
Set $\chi =\kappa  \xi .
$
For $s>0$, set $\varphi_s(z)=s(|z|^2-\frac{r_1^2}{4}).$

We now estimate the integrations on the left hand side of \eqref{eq:levi-test}  on $B_{r_1/2}$ 
and $\mathbb{D}_{r} \backslash B_{\frac{r_1}{2}}$ separately,
with $\chi$ and $\psi$ replaced by $\kappa \xi$ and $\varphi_s$ respectively.

\begin{itemize}
\item
On $B_{\frac{r_1}{2}}$, as $s\ra +\infty$, we have
we have
$$
\int_{B_{\frac{r_1}{2}}}|\dbar \chi|^2
+
\Delta p \cdot |\chi|^2
e^{-\varphi_s-p} d\lambda
\leq
\int_{B_{\frac{r_1}{2}}}(0-c)e^{-(\phi+\psi_s)} d\lambda.
$$
\item
On $\mathbb{D}_r \backslash B_{\frac{r_1}{2}}$,
notice 
that $\lim_{s\rightarrow+\infty}\varphi_s(z)=+\infty$ for $z\in \mathbb{D}_r \backslash B_{\frac{r_1}{2}}$.
We have 
$$
\lim_{s\rightarrow+\infty}\left(
  \int_{\mathbb{D}_r \backslash B_{\frac{r_1}{2}}}|\dbar \chi|^2
+
\Delta p \cdot |\chi|^2
e^{-\varphi_s-p} d\lambda\right)=0.
$$
\end{itemize}

Notice that $\varphi_s\leq 0$ on $B_{r/2}$,
we have,
 $$
 \int_{\mathbb D_r}
\left(
|\dbar \chi|^2
+
\Delta p \cdot |\chi|^2
\right)
e^{-\varphi-p}d\lambda<0
 $$ for $s>>1$,
which contradicts to the inequality \eqref{eq:levi-test}.

Since \(L\) is arbitrary, 
then \(\psi_{\varepsilon}\) is plurisubharmonic on \(D_{\varepsilon}\).
\end{proof}

\section{Proof of Theorem \ref{thm:main-1}}\label{sec:4}

We first recall the following analytic characterization of pseudoconvexity.

\begin{thm}[\cite{Liu-Zhang},\cite{Liu-Zhang-Skoda}]\label{thm:liu-zhang-char-domain}

Let $D$ be a domain in $\mathbb{C}^n$ satisfies $\mathring{\overline{D}}=D$, and let $\psi\in C^{0}(D,\mathbb{R})$ be a real-valued continuous function defined on $ D$,
assume that $(D,\psi)$ is $L^2$-optimal, 
then $D$ is pseudoconvex domain.
\end{thm}

Thus, associated with every fixed \(L^2\)-optimal pair \((D,\varphi)\), there is a natural pseudoconvex exhaustion of \(D\). Since \(D\) is bounded, we may choose the domains \(D_{\varepsilon}\) appearing in Proposition~\ref{prop:main-mollifier} to form such a pseudoconvex exhaustion of \(D\), with
$$
\bigcup_{\varepsilon>0} D_{\varepsilon}=D
$$
where $\{D_{\epsilon}\}$ is a sequence of pseudoconvex domains satisfy \(D_{\varepsilon_1}\subset D_{\varepsilon_2}\) whenever \(0<\varepsilon_2<\varepsilon_1\).

\begin{prop}\label{prop:decreasing-psh-approx}
Set $\delta_j:=4^{-j}.$
Then there exists a sequence \(\varepsilon_j\downarrow0\) and smooth
plurisubharmonic functions
\[
q_j:=p_{\varepsilon_j}+3\delta_j,
\]
defined on \(D_{\varepsilon_j}\) such that $q_{j+1}<q_j$  on $D_{\varepsilon_j}.$
\end{prop}

\begin{proof}
By the local
uniform convergence,
we have $p_\varepsilon\longrightarrow\psi
$ when $\varepsilon\downarrow0.$

Thus there exits a sequence $\{\varepsilon_j\}$ with $\varepsilon_{j}\downarrow 0$
and 
\(p_{\varepsilon_j}\) is defined on 
\(D_{\varepsilon_j}\) and
\begin{equation}\label{eq:peps-approx}
\left|p_{\varepsilon_j}-\psi\right|
<
\delta_j
\qquad
\text{on } D_{\varepsilon_j}.
\end{equation}
Set $\varepsilon_{j+1}<\varepsilon_j,
$ such that $D_{\varepsilon_j}
\subset
D_{\varepsilon_{j+1}}.$

Since $q_j-\psi
=
p_{\varepsilon_j}-\psi+3\delta_j,$
then \eqref{eq:peps-approx} shows  that
\[
2\delta_j
<
q_j-\psi
<
4\delta_j
\qquad
\text{on }\overline D_{\varepsilon_j}.
\]

By the monotonicity of $\{D_{j}\}$,
since 
$$
4\delta_{j+1}
=
4\cdot 4^{-(j+1)}
=
\delta_j,
$$
the estimate 
$$
2\delta_{j+1}
<
q_{j+1}-\psi
<
4\delta_{j+1}
\qquad
\text{on }D_{\varepsilon_j+1}
$$
is also holds on
\(D_{\varepsilon_j}\). 

Thus we have 
\[
q_{j+1}-\psi
<
4\delta_{j+1}
=
\delta_j
<
2\delta_j
<
q_j-\psi.
\]
holds on on \( D_{\varepsilon_j}\).
Hence $q_{j+1}<q_j $ on $
D_{\varepsilon_j}$.

\end{proof}

\begin{prop}[Example 7.4 of Chapter 1 in \cite{Demaily}]
\label{prop:hartogs-domain}
Let $\omega\subset\mathbb C^{n}$ be a pseudoconvex domain, and let $u$
be a real valued, upper semicontinuous function. Consider the Hartogs domain
\[
\Omega
=
\left\{
(z_1,z')\in \omega \times \mc:
u(z')<e^{-u}
\right\}.
\]
Then $\Omega$ is pseudoconvex if and only if $u$ is
plurisubharmonic on $\omega$.
\end{prop}

We can now can prove Theorem~\ref{thm:main-1}.
\begin{thm}[= Theorem \ref{thm:main-1}]
  Let $D$ be a bounded domain in $\mathbb{C}^n$ satisfies $\mathring{\overline{D}}=D$, and let $\psi\in C^{0}(D,\mathbb{R})$ be a real-valued continuous function defined on $ D$,
assume that $(D,\psi)$ is $L^2$-optimal, the Hartogs domain
\[
D_\psi
=
\{(z,w)\in D\times\C:|w|<e^{-\psi(z)}\}
\]
is a pseudoconvex domain, furthermore, $\psi$ is plurisubharmonic. 
\end{thm}

\begin{proof}
By Proposition \ref{prop:decreasing-psh-approx},
there exist a sequence of pair 
$(D_{\epsilon_j}, q_{\epsilon_j})$ such that 
$$
\Omega_{j}
=
\{(z,w)\in D_{\epsilon_j}\times\C:|w|<e^{-q_{\epsilon_j}(z)}\}
$$
is an increasing sequence of domain in $\mc^{n+1}$ satisfying 
$\Omega_j\subset\Omega_{j+1}$,
and Proposition \ref{prop:hartogs-domain}
shows that $\Omega_j$ is a pseudoconvex domain,
Furthermore, by the monotonicity and the convergence
of $q_{\epsilon_j}$,
we have 
\[
D_\psi
=
\bigcup_{j\geq1}\Omega_j,
\]
the Behnke--Stein
increasing-union theorem therefore implies that $D_\psi$ is
pseudoconvex,
thus $\psi$ is a plursubharmonic function.
\end{proof}

\end{document}